\documentclass[11pt,reqno,a4paper]{amsart}

\usepackage{amsfonts}
\usepackage{epsfig}
\usepackage{color}
\usepackage{graphicx}
\usepackage{amsmath}
\usepackage{amssymb}
\DeclareMathAlphabet{\mathpzc}{OT1}{pzc}{m}{it}
\newcommand{\Rr}{\mathbb{R}}

\newcommand{\EE}{\mathcal{E}}

\newcounter{main}

\numberwithin{equation}{section}

\newtheorem{theorem}{Theorem}[section]
\newtheorem{proposition}[theorem]{Proposition}

\newtheorem{corollary}[theorem]{Corollary}

\newtheorem{remark}{Remark}[section]

\newtheorem{maintheorem}{Theorem}

\newcommand{\blanksquare}{\,\,\,$\sqcup\!\!\!\!\sqcap$}

\newcounter{example}
{{\stepcounter{example}}{\flushleft {\bf Example \arabic{example}:}}}%
{\par}

\title[On the entropy of conservative flows]
{On the entropy of conservative flows}

\author{M\'ario Bessa and Paulo Varandas}
\address{M\'ario Bessa, Centro de Matem\'atica da
  Universidade do Porto, Rua do Campo Alegre 687, 4169-007
  Porto, Portugal \\ ESTGOH-Instituto Polit\'ecnico de Coimbra, Rua General Santos Costa, 3400-124 Oliveira do Hospital, Portugal}
\email{bessa@fc.up.pt}
\address{Paulo Varandas, Departamento de Matem\'atica, Universidade Federal da Bahia\\
Av. Ademar de Barros s/n, 40170-110 Salvador, Brazil.}
\email{paulo.varandas@ufba.br}

\begin{document}

\begin{abstract}
We obtain a $C^1$-generic subset of the incompressible flows in a closed three-di\-men\-sional manifold  where Pesin's entropy formula holds thus establishing the con\-ti\-nu\-ous-time version of  \cite{T}. Moreover, in any compact manifold of dimension larger or equal to three we obtain that the \emph{metric entropy function} and the \emph{integrated upper Lyapunov exponent  function} are not continuous with respect to the $C^1$ Whitney topology. Finally, we establish the $C^2$-genericity of Pesin's entropy formula in the context of Hamiltonian four-dimensional flows.
\end{abstract}
\maketitle


\begin{section}{Introduction: basic definitions and some results}

\begin{subsection}{Notation and basic definitions}
We consider a three-di\-men\-sio\-nal closed and connected $C^\infty$ Riemannian manifold $M$ endowed with a volume-form. Let $\mu$ denote the measure associated to it that we call Lebesgue measure. We say that a vector field $X\colon M\rightarrow TM$ is \emph{divergence-free} if $\nabla\cdot X=0$ or equivalently if the measure $\mu$ is invariant for the associated flow, $X^t\colon M\rightarrow M$, $t \in \mathbb{R}$. In this case we say that the flow is \emph{incompressible} or \emph{volume-preserving}. We denote by $\mathfrak{X}_\mu^r(M)$ ($r\geq 1$) the space of $C^r$ divergence-free vector fields on $M$ and we endow this set with the usual $C^r$ Whitney topology. Denote by $dist(\cdot,\cdot)$ the distance in $M$ inherited by the Riemannian structure. 
Given $X\in\mathfrak{X}_\mu^{1}(M)$ let $Sing(X)$ denote the set of \emph{singularities} of $X$ and $\mathpzc{R}:=M\setminus Sing(X)$ the set of \emph{regular} points. 

A vector field is said to be \emph{Anosov} if the tangent bundle $TM$
splits into three continuous $DX^t$-invariant nontrivial subbundles
$E^0\oplus E^1\oplus E^2$ where $E^0$ is the flow direction, the
sub-bundle $E^2$ is uniformly contracted by $DX^{t}$ and the
sub-bundle $E^1$ is uniformly contracted by $DX^{-t}$
for all $t>0$. Of course that, for an Anosov flow, we have $Sing(X)=\emptyset$ which follows from the fact that the dimensions of the subbundles are constant on the entire manifold.
\end{subsection}

\begin{subsection}{Lyapunov exponents, Oseledets' theorem and the integrated exponent function}
Since $X^t$ is incompressible we can apply Oseledets' multiplicative ergodic theorem ~\cite{O} to the 
volume-pre\-ser\-ving diffeomorphism $f=X^1$ and obtain for $\mu$-a.e. point $x\in{M}$, a splitting
$T_{x}M=E^{1}_{x}\oplus...\oplus{E^{k(x)}_{x}}$ (\emph{Oseledets splitting}) and real numbers $\lambda_{1}(x)>...>\lambda_{k(x)}(p)$
(\emph{Lyapunov exponents}) such that $Df_{x}(E^{i}_{x})=E^{i}_{f(x)}$
and $$\underset{n\rightarrow{\pm{\infty}}}{\lim}\frac{1}{n}\log{\|Df^{n}_{x}\cdot v^{i}\|=\lambda_{i}(x)},$$
for any $v^{i}\in{E^{i}_{x}\setminus\{\vec{0}\}}$ and $i=1,...,k(x)$. Using the three-dimension and the conservativeness assumptions we observe that $k(x)=1$ or $k(x)=3$ (see the paragraph after \eqref{det} below). If $k(x)=1$ then the spectrum is trivial, that is, all the Lyapunov exponents vanish.

We denote the $\mu$-a.e. points given by this theorem by $\mathcal{O}(X^1)=\mathcal{O}(X)$. It is clear that, fixing $t\in\mathbb{R}$, any $g=X^t$ is such that $\mathcal{O}(g)=\mathcal{O}(X^1)$.

We can obtain a proof of Oseledets' theorem for the flow dynamics using its discrete version, let us see briefly how:  Since it is an asymptotic result and $DX^{r}_{x}$, for fixed $r$ and $x$ varying in a compact set, is a uniformly bounded operator we may replace the tangent map
$DX^{t}_{x}=DX^{r}_{X^{n}(x)}\circ{DX^{n}_{x}}$ by the least
integer time-$n$ map, $DX^{n}_{x}$, and compute the limit as before. 
The Oseledets splitting associated to the flow $X^t$ on any point along the orbit of $x\in\mathcal{O}(X^1)$ is the saturation, by the tangent flow $DX^t$, of the directions given by $T_{x}M=E^{1}_{x}\oplus...\oplus{E^{k(x)}_{x}}$. This theorem allows us to conclude also that
\begin{equation}\label{angle}
\lim_{t\rightarrow{\pm{\infty}}}\frac{1}{t}\log{|\det(DX^{t}_{x})|=\sum_{i=1}^{k(x)}\lambda_{i}(x)\dim(E^{i}_{x})},
\end{equation}
which is related to the sub-exponential decrease of the angles
between any subspaces of the Oseledets splitting along $\mu$-a.e.
orbit. 
Since we have the invariance $DX^{t}_{x}(X(x))=X(X^{t}(x))$, we conclude that one of Oseledets' subspaces is $E^0(x)$, and  that its associated Lyapunov exponent is zero. 

By the \emph{Liouvile formula} 
$$
\det DX^\tau_x=e^{\int_0^\tau \nabla \cdot X(X^t(x))dt},
$$
we get that whenever $\nabla\cdot X=0$ then 
\begin{equation}\label{det}
|\det(DX^{t}_{x})|=1,\,\,\, \forall t\in\mathbb{R}.  
\end{equation}

Since we are in a three-dimensional setting, then using~(\ref{angle}) and (\ref{det}), we have
$\lambda_{1}(x)+\lambda_{3}(x)=0$. Hence either
$\lambda_{1}(x)=-\lambda_{3}(x)>0$ or both Lyapunov exponents are equal to zero. In the former
case 
there exists a decomposition of the tangent space $T_x M$ as direct sum of 
two subspaces $E^{+}_{x}$ and $E^{-}_{x}$ associated to the positive Lyapunov
exponent $\lambda_{1}(x)=\lambda_{+}(x)$ and the negative one $\lambda_{3}(x)=\lambda_{-}(x)$, respectively. 

We usually reduce the study of hyperbolicity to the orbit normal bundle. Given $x\in \mathpzc{R}$ we consider its normal bundle $N_{x}=X(x)^{\perp}\subset T_{x}M$ and define the \emph{linear Poincar\'{e} flow} by $P_{X}^{t}(x):=\Pi_{X^{t}(x)}\circ DX^{t}_{x}$ where $\Pi_{X^{t}(x)}:T_{X^{t}(x)}M\rightarrow N_{X^{t}(x)}$ is the projection along the direction of $X(X^{t}(x))$. 

Due to the aforementioned property of sub-exponential decrease of the angles
between any Oseledets subspaces, it is not hard to check that if $E^\sigma_x$ ($\sigma=+/-$) is associated to the Lyapunov exponent $\lambda^\sigma(x)$, then $N^\sigma_x:=\Pi_{x}E^\sigma_x$ is an Oseledets invariant subspace of $N^{\sigma}_x$. Furthermore, the Lyapunov exponents are given by
$$
\underset{t\rightarrow{\pm{\infty}}}{\lim}\frac{1}{t}\log{\|P^{t}_{X}(x)|_{N^{\sigma}_x}}\|=\underset{t\rightarrow{\pm{\infty}}}{\lim}\frac{1}{t}\log{\|DX^t_{x}|_{E^{\sigma}_x}}\|=\lambda^{\sigma}(x).
$$

Let us consider the following integrated upper Lyapunov exponent function:
$$
\begin{array}{cccc}
\Lambda\colon & \mathfrak{X}^{1}_{\mu}(M) & \longrightarrow & \mathbb{R}\\
& X & \longmapsto & \int_M \lambda^{+}(X,x)d\mu(x).
\end{array}
$$

The next simple equality was proved in~\cite[Proposition 2.1]{B}:
\begin{equation}\label{inf}
\Lambda(X)=\underset{n\geq{1}}{\inf}\frac{1}{n}\int_{M}\log\|P^{n}_{X}(x)\|d\mu(x).
\end{equation}

Notice that the function $\Lambda$ is the infimum of continuous functions, hence upper semicontinuous. In particular, 
the continuity points of $\Lambda$ is a residual subset of $\mathfrak{X}^{1}_{\mu}(M)$  (dense $G_\delta$ in the $C^1$ topology).

The following result was proved by the first author for the case of incompressible flows on three-dimensional closed manifolds without singularities (see~\cite[Proposition 3.2]{Be}) and then generalized for the context admitting singularities in ~\cite[Proposition 2.2]{AB}. Recall that $X$ is said to be an \emph{aperiodic} vector field if the Lebesgue measure of the set of periodic points and singularities is zero. Moreover, given $\ell \in \mathbb N$
we say that the splitting $N=N^-\oplus N^+$ of the normal bundle over an invariant set $\Lambda$ is an $\ell$-\emph{dominated splitting} for the linear Poincar\'{e} flow if there exists an $\ell \in {\mathbb N}$ such that for all $x\in \Lambda$  we have:
$$
\|P_{X}^{\ell}(x)|_{N^{-}_{x}}\|  . \|P_{X}^{-\ell}(X^\ell(x))|_{N^{+}_{X^\ell(x)}}\|   \leq \frac{1}{2}.
$$
For simplicity reasons we refer to $\Lambda$ as an $\ell$-dominated invariant set.

\begin{theorem}\label{teorema1}
Let $X\in{\mathfrak{X}_{\mu}^{2}(M)}$ be an aperiodic vector field and assume that every $\ell$-dominated invariant subset has zero volume. For every given $\epsilon,\delta>0$ there exists a vector field $Y\in {\mathfrak{X}_{\mu}^{1}(M)}$ such that $Y$ is $\epsilon$-$C^1$-close to $X$ and $\Lambda(Y)<\delta$.
\end{theorem}

As a consequence we obtain the following result:

\begin{theorem}\label{bessa}
There exists a residual subset $\mathcal{R}$ of $\mathfrak{X}_{\mu}^{1}(M)$, such that if $X\in\mathcal{R}$ is not an Anosov flow, then Lebesgue a.e. point in $M$ has zero Lyapunov exponents.
\end{theorem}

\end{subsection}

\begin{subsection}{Measure-theoretic entropy for flows, Margulis-Ruelle inequality and Pesin's formula}

Given $X\in\mathfrak{X}_{\mu}^{1}(M)$, we say that the associated flow $X^t\colon M\rightarrow M$ is \emph{expansive} if given any $\epsilon>0$, there exists $\delta>0$ such that, if $dist(X^t(x),X^{\tau(t)}(y))<\delta$ for all $t\in\mathbb{R}$, for all $x,y\in M$ and for all continuous maps $\tau\colon\mathbb{R}\rightarrow\mathbb{R}$, then $y=X^t(x)$ where $|t|<\epsilon$. 
This definition, introduced by Bowen and Walters in ~\cite{BW} roughly means that any two points whose orbits 
by the flow remain indistinguishable up to any continuous time displacement lie in the same orbit.  Moreover,
expansiveness is a topological invariant and any expansive flow admits at most countably many periodic orbits. 
%(see \cite{BW}).

\begin{remark}\label{Anosov}
The Anosov flows are expansive (see~\cite{An}).
\end{remark}

We give a brief description of the key concept of entropy, introduced in the theory of dynamical systems by Kolmogorov more than fifty years ago. In fact, topological entropy is 
one of the most important invariants in dynamics and describes the topological complexity of the system 
measuring how the dynamics separates and spread under iteration. In compact metric spaces topological entropy coincides with the limiting of the measure theoretical entropies described below, related with many other fundamental
concepts in dynamics as the Lyapunov exponents or Hausdorff dimension. We refer the reader to~\cite{Ka} for a very complete exposition on entropy.
In our volume-preserving setting, to understand the underlying dynamics it becomes relevant to study the
measure theoretical entropy with respect to the Lebesgue measure and its relation with Lyapunov exponents, which gave us the starting point for our study.  First we recall some definitions.

Given a vector field $X\in\mathfrak{X}_{\mu}^{1}(M)$ we define its measure-theoretic entropy, $h_\mu(X)$, by $h_{\mu}(X^1)$ where $X^1$ is the time-one of its associated flow. The following result is due to Abramov~\cite{A}, for a proof see ~\cite[Theorem 3 pp. 255]{CFS}.
\begin{theorem}
The metric entropy of the time-$t$ map $X^t$ is $|t|h_{\mu}(X^1)$ for any $t\in\mathbb{R}$. 
\end{theorem}

It is worth to point out that Sun and Vargas \cite{SV} defined a different concept of flow entropy which is well behaved when we consider a re-parametrization of the flow. 

Given a measure space $\Sigma$, a map $R\colon \Sigma\rightarrow{\Sigma}$, an $R$-invariant probability
measure $\tilde{\mu}$ defined in $\Sigma$ and a ceiling function
$h\colon \Sigma\rightarrow{\mathbb{R}^{+}}$ satisfying $h(x)\geq{\alpha}>0$
for all $x\in{\Sigma}$ and the integrability condition
%\begin{equation}\label{over}
 $
 \int_{\Sigma}h(x)d\tilde{\mu}(x)<\infty,
$
%\end{equation}
consider the space $M_{h}\subseteq{\Sigma\times{\mathbb{R}_+}}$ defined by
$$
M_h=\{(x,t) \in \Sigma\times{\mathbb{R}_+}: 0 \leq t \leq h(x) \}
$$
with the identification between the pairs $(x,h(x))$ and $(R(x),0)$. The semiflow defined
on $M_h$ by $S^s(x,r)=(R^{n}(x),r+s-\sum_{i=0}^{n-1}h(R^{i}(x)))$, 
where $n\in{\mathbb{N}_0}$ is uniquely defined by 
$$
\sum_{i=0}^{n-1}h(R^{i}(x))\leq{r+s}<\sum_{i=0}^{n}h(R^{i}(x))
$$ 
is called a \emph{suspension semiflow}. If $R$ is invertible then $(S^t)_t$ is indeed a flow.
Furthermore, if $\text{Leb}_1$ denotes the one dimensional Lebesgue measure it is not hard to check 
that the measure $\mu=(\tilde\mu \times \text{Leb}_1)/\int h\, d\tilde\mu$ defined on $M_h$
by
$$
\int g \, d\mu= \frac{1}{\int h\, d\tilde\mu} \int \left( \int_0^{h(x)} g(x,t) dt \right)\, d\tilde\mu(x),  \quad \forall g\in C^0(M_h)
$$
is a probability measure and it is invariant by the suspension semiflow $(S^t)_t$. In fact,
if $h$ is bounded then  
\begin{equation}\label{eq:identification}
\eta \mapsto \frac{\eta \times \text{Leb}_1}{\int h d\eta}
\end{equation}
is a  one-to-one correspondence between $R$-invariant probability measures and $S^t$-invariant probability measures (see e.g. \cite{BR}). In addition, it follows from the previous integrability condition and \emph{Abramov's formula} (AF) that
\begin{equation}\label{entropy}
h_\mu(S^t)\overset{(AF)}{=}\frac{h_{\tilde{\mu}}(R)}{\int_\Sigma h d\tilde{\mu}}.
\end{equation}
By abuse of notation, given a vector field $X\in \mathfrak{X}^1_\mu$ we let $h_\mu(X)$ denote the entropy with 
respect to the volume $\mu$ of the flow $X^t$ associated to $X$.
The following result, due to Bowen and Walters \cite[Theorem 6]{BW}, will be useful in the sequel.

\begin{theorem}\label{BW}
Let $X^t$ be a flow and $S^t$ the suspension flow representation of $X^t$ with section $\Sigma$, ceiling function $h$ and return map $R$. The flow $X^t$ is expansive if and only if $R$ is expansive. 
\end{theorem}

\begin{theorem}\label{Pesin}(Pesin's entropy formula for flows) If $X\in\mathfrak{X}_{\mu}^{1+\alpha}(M)$ with $\alpha>0$, then $h_{\mu}(X)=\Lambda(X)$. 
 
\end{theorem}
\begin{proof}
The proof is straightforward. Since $X^1$ is a $C^{1+\alpha}$ volume-preserving diffeomorphism we apply Pesin's entropy formula (see~\cite{P,M}) and obtain $h_{\mu}(X^1)=\int_M \lambda^{+}(X^1,x)d\mu(x)$. 
By the definition of entropy for the flow we have $h_{\mu}(X^1)=h_{\mu}(X)$. Finally, the result follows from the
fact that the upper Lyapunov exponent of the flow is equal to the one associated to the time-one map. 
\end{proof}

Using Margulis-Ruelle's inequality for discrete time systems \cite{R} we obtain analogously:

\begin{theorem}\label{MR} If $X\in\mathfrak{X}_{\mu}^{1}(M)$, then $h_{\mu}(X)\leq\int_M \lambda^{+}(X,x)\,d\mu(x)$. 
 
\end{theorem}

\end{subsection}

\end{section}

\begin{section}{Statements and proof of the results}

In the next proposition we cannot use directly ~\cite[Proposition 1.5]{T} because the time-one of an Anosov flow is not an Anosov diffeomorphism.  Actually, it is a partially hyperbolic diffeomorphism (see~\cite{BDV} for the definition).
Moreover, it is still unknown (Verjovski conjecture) whether every volume-preserving Anosov flow
is the suspension of an Anosov diffeomorphism. Nevertheless, we could prove the semicontinuity of the metric
entropy for the three-dimensional Anosov incompressible flows.

\begin{proposition}\label{R1}
The set of continuity points of $h_\mu$ when restricted to the Anosov incompressible flows is $C^1$-residual within three-dimensional Anosov incompressible flows. 
\end{proposition}

\begin{proof}
It follows from \cite{Bo73,Rt} that any Anosov flow $X^t$  admits a finite Markov partition of arbitrary small size. 
Moreover, there exists a subshift of finite type $\sigma\colon\Sigma \to\Sigma$ and a ceiling function $h$ with
summable variation so that $X^t$ is semiconjugated to the symbolic suspension flow $S^t$ as above with $R=\sigma$ and $M=\Sigma$, and a bounded continuous roof function $h$ . 
More precisely,  following~\cite{Bo73}, there exists a finite-to-one continuous surjection $\phi\colon \Sigma_h \to M$ satisfying $\phi\circ S^t=X^t \circ \phi$ for every $t\in \mathbb R$. Indeed, $\phi$ is a bijection on the complement
of the set $\phi^{-1}(\cup_{t\in\mathbb R} \partial \mathcal{P})$, where $\mathcal{P}$ are the rectangles of the Markov partition in $\Sigma$. For completeness reasons let us point out that the construction of the Markov partitions use
a finite number of two-dimensional cross sections $\Sigma_0$ transverse to the vector field $X$, that can be taken uniform for every $C^1$-close vector field.   
Finally, since we deal with two-dimensional cross-sections, the boundaries of the Markov partition are formed
by a finite union of one dimensional smooth curves obtained as intersection of two-dimensional $W^{cs}$ and $W^{cu}$ manifolds with the sections $\Sigma_0$ which form a zero Lebesgue measure set.

In particular, $\mu_{\Sigma_h}=\phi^*\mu$ is a well defined $S^t$-invariant probability measure on $\Sigma_h$ and $\phi$ is a measure theoretical isomorphism between $(X^t,\mu)$ and $(S^t,\mu_{\Sigma_h})$. Hence,
using~\eqref{eq:identification} and \eqref{entropy}  one deduces that there exists a $\sigma$-invariant probability measure
$\mu_{\Sigma}$ such that
$$
h_\mu(X^t)=h_{\mu_{\Sigma_h}}(S^t) = \frac{h_{\mu_{\Sigma}}(\sigma)}{\int h \,d\mu_\Sigma}.
$$
Recall that every Anosov flow is expansive (see Remark~\ref{Anosov}) and the expansiveness constant varies
continuously within Anosov flows. Hence, there exists a uniform $\epsilon>0$ for which all $Y\in \mathfrak{X}^1_\mu(M)$ that are $C^1$-close to $X$ are $\epsilon$-expansive and, by Theorem~\ref{BW}, the base maps 
$\sigma_Y$ are also (uniformly) expansive. So, let $\mathcal P$ be a partition on $\Sigma_0$ that is generating
for all $\sigma_Y$. It follows from Kolmogorov-Sinai's theorem and sub-additivity that
$$
h_{{\mu}_\Sigma}(\sigma,\mathcal{P})=\underset{n\rightarrow +\infty}{\lim}\frac{1}{n}H_{\mu_\Sigma}(\mathcal P^{(n)})
= \inf_{n\ge 1}\frac{1}{n}H_{\mu_\Sigma}(\mathcal P^{(n)}),
$$
where $\mathcal{P}^{(n)}=\mathcal{P}\vee \sigma^{-1}(\mathcal{P})\vee ... \vee \sigma^{-n+1}(\mathcal{P})$ is the dynamically refined partition on $\Sigma$ and $H_{{\mu}_\Sigma}(\mathcal Q)=\sum_{Q\in \mathcal Q} -\mu_\Sigma(Q)\log \mu_\Sigma(Q)$ for every partition $\mathcal Q$ on $\Sigma$. 
Notice that the function 
$$
X \mapsto \frac{1}{n} H_{\mu^X_{\Sigma_h}} (\mathcal{P}^{(n)})
$$
is continuous. As a consequence, the function $X\mapsto \underset{n\in\mathbb{N}}{\inf}\frac{1}{n} H_{\tilde{\mu}_X}(\mathcal{P}_X^{(n)})$ is upper semi-continuous because it is the infimum of continuous functions. Hence, there exists a residual subset such that 
the previous function is continuous.
\end{proof}

\begin{proposition}\label{R2}
Let $\mathcal{R}$ be the residual given by Theorem~\ref{bessa} and let $\mathcal{Z}$ stands for the set $\mathcal{R}$ except the $C^1$ closure of the Anosov flows. Then, any $X\in \mathcal{Z}$ is a continuity point of the metric entropy function $h_{\mu}$.
\end{proposition}

\begin{proof}
Let be given $\epsilon>0$. Is is sufficient to prove that exists $\delta>0$, such that any divergence-free vector field $Y$ $\delta$-$C^1$-close to $X$ satisfies $h_{\mu}(Y)\leq \epsilon$. By Theorem~\ref{MR} we have $h_{\mu}(Y)\leq\int_M \lambda^{+}(Y,x)d\mu(x)=\Lambda(Y)$. Since $\Lambda$ is upper semicontinuous we have $\Lambda(Y)\leq \Lambda(X)+\epsilon$ for some $\epsilon>0$ small. But $\Lambda(X)=0$ because $X\in \mathcal{Z}$ and so $h_{\mu}(Y)\leq \epsilon$.
\end{proof}

Using Proposition~\ref{R1} and Proposition~\ref{R2} we obtain the following result.

\begin{corollary}\label{C3}
The set of continuity points of $h_\mu$ is a $C^1$-residual in $\mathfrak{X}_{\mu}^{1}(M)$. 
\end{corollary}

The following result is the continuous-time counterpart of Tahzibi's theorem~\cite{T}.

\begin{maintheorem}\label{teo1}
The Pesin entropy formula holds for a $C^1$-residual subset of $\mathfrak{X}_{\mu}^{1}(M)$.
\end{maintheorem}
\begin{proof}
Recalling the paragraph after (\ref{inf}) and using Corollary~\ref{C3} we conclude that there exists a residual $\mathcal{R}\subset \mathfrak{X}^1_{\mu}(M)$ such that any $X\in \mathcal{R}$ is a continuity point of both $h_{\mu}$ and $\Lambda$. By Zuppa's theorem~\cite{Z} we know that $\mathfrak{X}^2_{\mu}(M)$ is $C^1$ dense in $\mathfrak{X}^1_{\mu}(M)$ and so we can take a sequence of $C^2$ divergence-free vector fields $\{X_n\}_{n\in\mathbb{N}}$ such that $X_n$ converge in the $C^1$ topology to $X$.

By Theorem~\ref{Pesin} since $X_n\in\mathfrak{X}_{\mu}^{2}(M)$, then $h_{\mu}(X_n)=\Lambda(X_n)$ for every $n\in\mathbb{N}$. Finally, given any $X\in\mathcal{R}$ we use the continuity of $h_{\mu}$ and $\Lambda$ at $X$ to conclude that $h_{\mu}(X)=\Lambda(X)$,  and the theorem is proved. 

\end{proof}

There are several results on the continuity of the metric and topological entropies over hyperbolic flows (see e.g.~\cite{Con} and the references therein). However, the next result shows that we cannot expect the continuity of the metric entropy in general.

\begin{maintheorem}\label{teo2}
Let $M$ be any compact manifold of dimension larger or equal to three. 
Then the functions $h_\mu\colon (\mathfrak{X}^1_{\mu}(M),C^1)\rightarrow \mathbb{R}$ and $\Lambda\colon 
(\mathfrak{X}^1_{\mu}(M),C^1)\rightarrow \mathbb{R}$ are not continuous.
\end{maintheorem}
\begin{proof}
By Hu, Pesin and Talitskaya theorem (see~ \cite{HPT}) we know that every compact manifold carries a 
$C^\infty$ volume-preserving flow $X^t$ associated to a divergence-free vector field $X\colon M\rightarrow TM$ with 
a nonuniformly hyperbolic Bernoulli ergodic component $\mu$. Moreover, although $\text{Sing}(X)\neq \emptyset$ the singularities have zero Lebesgue measure and, by nonuniform hyperbolicity and Pesin's formula, 
$h_\mu(X)=\Lambda(X)>0$.
We claim that $X$ is a discontinuity point for both functions $h_{\mu}$ and $\Lambda$. Indeed, using that $X^t$ is
not Anosov it follows from Theorem~\ref{bessa} and the fact that $(\mathfrak{X}^1_{\mu}(M),C^1)$ is a Baire space 
that there exists a $C^1$-residual subset in $\mathcal R \subset \mathfrak{X}^1_{\mu}(M)$,  such that any $Y\in \mathcal{R}$ is such that $h_\mu(Y)\leq \Lambda(Y)=0$. This proves that  $X$ is not a continuity point for $h_\mu$ or $\Lambda$ and finishes the proof of the theorem.
\end{proof}

An interesting question is to characterize the continuity points of the metric entropy and Lyapunov exponent functions.  

\medskip

\noindent\textbf{Question 1:} Does the discontinuity points form a dense set? 

\end{section}

\begin{section}{Hamiltonians in symplectic $4$-manifolds}

It is worth mention that Theorem~\ref{teo1} has a counterpart for four-dimensional $C^2$-Hamiltonian systems.
In fact, the arguments used in the proof of the Theorem~\ref{teo1} lie the dichotomy of hyperbolicity versus almost 
everywhere zero Lyapunov exponents which was extended to the setting of four-dimensional $C^2$-Hamiltonian systems  in \cite{BD}. 

Let us recall some elementary facts about Hamiltonians. Let $M$ be a compact symplectic four-dimensional manifold endowed with the symplectic two-form $\omega$. We will be interested on \emph{Hamiltonians on $M$}, i.e., real functions on $M$ endowed with the $C^2$-topology. Given a Hamiltonian $H\colon M\rightarrow\mathbb{R}$, any scalar $e\in H(M)\subset\Rr$ is called an \emph{energy of $H$} and any connected component of 
$H^{-1}(e)=\{x\in M\colon H(x)=e\}$ the corresponding invariant energy level set. 
It is \emph{regular} if it does not contain critical points. For any $C^2$ Hamiltonian function $H\colon M\to{\mathbb{R}}$ there is a corresponding \emph{Hamiltonian vector field} $X_H\colon M\to TM$ determined by $\omega(X_H,\cdot)=DH(\cdot)$. Observe that $H$ is $C^2$ if and only if $X_H$ is $C^{1}$. The Hamiltonian vector field generates the \emph{Hamiltonian flow}, a smooth 1-parameter group of symplectomorphisms $\varphi^{t}_{H}$ on $M$. 
 
The volume form $\omega^2$ gives a measure $\mu$ on $M$ that is preserved by the Hamiltonian flow. On each regular energy surface $\EE\subset M$ (of dimension three) there is a natural finite invariant volume measure which we denote by $\mu_\EE$. We define hyperbolicity and also compute the Lyapunov exponents associated to the linear Poincar\'e flow of the Hamiltonian flow restricted to each $\EE$. Thus this \emph{transversal linear Poincar\'e flow} has dimension two and so let $\lambda^+(H,x)$ denotes the largest (or the nonnegative) Lyapunov exponent associated to the flow $\varphi^{t}_{H}$ (we refer to \cite[\S 2]{BD} for the full details on these structures).
Consider also $h_{\mu_{\mathcal E}}(X_H)$ denote the measure theoretical entropy of the flow $\varphi_H^t$ on the
level set $\mathcal E_p(H)$, and set 
$$
h_\mu(X_H)= \int h_{\mu_{\mathcal E}}(X_H) dH.
$$

Since the dynamics when restricted to each regular energy surface is a three-dimensional flow we can recover an analog of Proposition~\ref{R1} for four-dimensional $C^2$-Hamiltonians. It is worth to observe that the uniform hyperbolicity (Anosov property) holds in some connected component of the energy surface (cf.~\cite[Definition 2.4]{BFR}).

Now we recall the following useful result.

\begin{theorem}\label{teorema22}\cite[Theorem 1]{BD}
Let $(M,\omega)$ be a $4$-dimensional compact symplectic manifold. 
For a $C^2$-generic Hamiltonian $H\in C^2(M,\Rr)$, the union of the regular energy surfaces $\EE$ that are either Anosov or have zero Lyapunov exponents $\mu_\EE$-a.e. for the Hamiltonian flow, forms an open $\mu$-mod $0$ and dense subset of $M$.
\end{theorem}

Consider the product space $\mathcal M=M\times C^2(M,\mathbb R)$ and the set
$$
\mathcal A=\left\{ (p,H) \in\mathcal M :  \mathcal E_p(H) \text{ is an Anosov regular level}\right\},
$$
where $\mathcal E_p(H) \subset H^{-1}(H(p))$ denotes the level set in $M$ containing the point $p$. When no
confusion is possible we write $\mathcal E$ for simplicity.  By \cite[Theorem 2]{BFR} $\mathcal A$ is open in $M\times C^2(M,\mathbb R)$. The next proposition
is the Hamiltonian version of Proposition~\ref{R2} and its proof is analogous.
\begin{proposition}
The set of continuity points of the function 
$$
\begin{array}{ccc}
\mathcal A & \to & \mathbb R \\ 
(p,H)  & \mapsto & h_{\mu_{\mathcal E}}(X_H)
\end{array}
$$
contains a residual subset $\mathcal R_1 \subset \mathcal A$.
\end{proposition}

Now we proceed to prove semicontinuity of the measure theoretical entropy among $C^2$-Hamiltonians. 
The following results are collected from~\cite{BD}. Set $\mathcal B=\mathcal M\setminus \overline{\mathcal A}$ where $\overline{\mathcal{A}}$ stands for the $C^2$ closure of the set $\mathcal{A}$.
 Then there exists a continuous function $\rho: \mathcal B \to \mathbb R^+$ such that the connected component
$\mathcal V_{p,H}\subset M$ of $\{x \in M: |H(x)-H(p)|<\rho(p,H)\}$ satisfies the following: given $\epsilon,\delta>0$
and  $(p,H)\in \mathcal B$  there exists a Hamiltonian $\tilde{H}$ that is $C^2$-$\epsilon$-close to $H$ and 
\begin{equation}\label{l.approx}
\int_{\mathcal V_{p,\tilde{H}}} \lambda^+(\tilde{H},x) \, d\mu(x)<\delta.
\end{equation}
Hence, the set 
$$
\left\{ (p,H) :  \int_{\mathcal V_{p,H}} \lambda^+(H,x) \, d\mu(x)=0 \right\}
$$
contains a residual subset of $\mathcal B$. Hence, proceeding as in the proof of Proposition~\ref{R2} to
bound the metric entropy by the integrated Lyapunov exponent we obtain the following:

\begin{proposition}
There exists a residual $\mathcal R_2\subset \mathcal B$ so that 
$\mathcal R_2 \ni (p,H) \mapsto h_{\mu_{\mathcal E}}(X_H)$ is continuous.
\end{proposition}

Moreover, following \cite[Proposition A.7]{BF} there exists a residual $\mathcal R \subset C^2(M,\mathbb R)$ and for every $H\in\mathcal R$ a residual subset $\mathcal R_H\subset M$ such
that 
$$
\mathcal R_1 \cup \mathcal R_2 = \bigcup_{H\in \mathcal R} \mathcal R_H\times \{H\}
$$
and, for every $H\in\mathcal R$ and $p\in \mathcal R_H$ either $\mathcal E_p(H)$ is Anosov or
$H$ has zero Lyapunov exponent in the following sense:
\begin{equation}
\int \int \lambda^+(H,x) \; d\mu_{\mathcal E}(x) dH=0.
\end{equation}
Observe that if $H\in C^3(M,\mathbb R)$ is Morse then $X_H$ is of class $C^2$ and it follows from Pesin's formula that 
$$
h_{\mu_{\mathcal E}}(X_H) = \int \lambda^+(H,x) \;d\mu_{\mathcal E}(x)
$$
for every regular level set $\mathcal E$ and we deduce
$$
h_{\mu}(X_H) = \int\int \lambda^+(H,x) \;d\mu_{\mathcal E}(x) \,dH.
$$

\begin{corollary}\label{C3H}
The set of continuity points of $h_\mu$ is a $C^2$-residual in the set of $C^2$ Hamiltonians in $M$. 
\end{corollary}

Recalling that $C^r$ Morse Hamiltonians ($r\geq 3$) are $C^2$-dense in our set of $C^2$-Hamiltonians we can obtain similarly the Hamiltonian counterpart of Theorem~\ref{teo1} above.

\begin{maintheorem}\label{teo3}
The Pesin entropy formula holds for a $C^2$-residual subset of the set of $C^2$ Hamiltonians in $M$. 
\end{maintheorem}

An interesting question is to know if the measure theoretical entropy and the integrated Lyapunov
exponent functions are continuous for Hamiltonians as in Theorem~\ref{teo2}. This would be true if the following
question has a positive answer.

\medskip
\noindent\textbf{Question 2:} In any symplectic manifold $M$ of dimension $2n+2$ ($n\in\mathbb{N}$) is there any Hamiltonian $H\colon M\rightarrow\mathbb{R}$ such that $H$ has $n$ positive Lyapunov exponents?

\end{section}

%%%%%%%%%%%%%%%%%%%%%%%%%%%%%%%%%%%%%%%%%
\section*{Acknowledgements}

We would like to thank Ali Tahzibi for his suggestions on these issues. M.B. was partially supported by FCT - Funda\c{c}\~ao para a Ci\^encia e a Tecnologia  through CMUP (SFRH/BPD/20890/2004). P.V was partially supported
by CNPq - Brazil.

%%%%%%%%%%%%%%%%%%%%%%%%%%%%%%%%%%%%%%%%%

\end{document}